\documentclass[11pt,a4paper,reqno]{amsart}

\usepackage{amsmath}
\usepackage{amsfonts}
\usepackage{graphicx}
\usepackage{framed}
\usepackage{amssymb,cancel}
\usepackage[margin=2.7cm]{geometry}
\usepackage{xcolor}
\usepackage{esint}
\usepackage{comment}

\usepackage{etex}
\usepackage{latexsym}
\usepackage[english]{babel}
\usepackage[latin1]{inputenc}
\usepackage[colorlinks=true,urlcolor=blue,
citecolor=red,linkcolor=blue,linktocpage,pdfpagelabels,
bookmarksnumbered,bookmarksopen]{hyperref}
\usepackage[hyperpageref]{backref}

\def \RN {\mathbb{R}^N}

\def \R {\mathbb{R}}

\def \e {\varepsilon}

\def \LL {\mathcal{L}}

\usepackage{amsthm}
\theoremstyle{definition}
\newtheorem{definition}{Definition}[section]

\newtheorem{remark}[definition]{Remark}

\theoremstyle{plain}
\newtheorem{theorem}[definition]{Theorem}
\newtheorem{proposition}[definition]{Proposition}
\newtheorem{lemma}[definition]{Lemma}

\numberwithin{equation}{section}

\makeatletter
\@namedef{subjclassname@2020}{\textup{2020} Mathematics Subject Classification}
\makeatother

\begin{document}

 \title[Local-Nonlocal Fujita type results]
 {Blow-up of solutions to  semilinear parabolic equations driven by mixed local-nonlocal operators \\ with large initial data}

 \author[S.\,Biagi]{Stefano Biagi}
 \author[F.\,Punzo]{Fabio Punzo}
 \author[E.\,Vecchi]{Eugenio Vecchi}

 \address[S.\,Biagi]{Dipartimento di Matematica
 \newline\indent Politecnico di Milano \newline\indent
 Via Bonardi 9, 20133 Milano, Italy}
 \email{stefano.biagi@polimi.it}

 \address[F.\,Punzo]{Dipartimento di Matematica
 \newline\indent Politecnico di Milano \newline\indent
 Via Bonardi 9, 20133 Milano, Italy}
 \email{fabio.punzo@polimi.it}

 \address[E.\,Vecchi]{Dipartimento di Matematica
 \newline\indent Università di Bologna \newline\indent
 Piazza di Porta San Donato 5, 40126 Bologna, Italy}
 \email{eugenio.vecchi2@unibo.it}



\thanks{The Authors are
member of the {\em Gruppo Nazionale per
l'Analisi Ma\-te\-ma\-ti\-ca, la Probabilit\`a e le loro Applicazioni}
(GNAMPA) of the {\em Istituto Nazionale di Alta Matematica} (INdAM). S.B. and E.V. are partially supported by the
Indam-GNAMPA project {\em Esistenza, regolarità e studio delle proprietà qualitative per problemi nonlineari}.\\}

 \begin{abstract}
We investigate finite-time blow-up for nonnegative solutions to the
Cauchy problem associated with semilinear parabolic equations driven
by a mixed local--nonlocal operator. The reaction term is assumed to
satisfy suitable structural hypotheses, the prototype being
$f(u)=u^p$ with $p>1$. By adapting the Kaplan method to the present
framework, we prove that solutions blow up in finite time whenever
the initial datum is sufficiently large. In the prototype case
$f(u)=u^p$, this conclusion holds for every $p>1$. As a particular case of our operator, we also include the fractional
Laplacian; to the best of our knowledge, this type of result is new
even in that special case.
 \end{abstract}

 \maketitle

\section{Introduction}\label{sec.Intro}
In this article, we continue the study initiated in \cite{BPV},
concerning the finite-time blow-up of nonnegative solutions to the
following Cauchy problem:
\begin{equation}\tag{$\mathrm{CP}_{\LL}$}\label{eq:mainPbParabolicIntro}
	\begin{cases}
	\partial_t u+\mathcal{L}	u = f(u) & \text{in $\mathbb{R}^N\times(0,+\infty)$} \\
	u(\cdot,0) = u_0 & \text{in $\mathbb{R}^N$},
	\end{cases}
\end{equation}
where $\mathbf{a} \geq 0$ and $\mathbf{b}>0$ are fixed parameters, $p\in(1,+\infty)$, 
$\mathcal{L}$ is the (weighted) mixed local-nonlocal operator defined as
$$\mathcal{L} = -\mathbf{a}\Delta + \mathbf{b}(-\Delta)^s.$$
Furthermore, as for the \emph{initial datum $u_0$}, we always assume that it is continuous, bounded, and non-negative on $\mathbb{R}^N$. Concerning the \emph{reaction term $f$},
we make the following structural assumption:
\begin{itemize}
	\item[$(H)_f$] $f:[0,+\infty)\to[0,+\infty)$ is a locally Lipschitz
	and
	convex function such that
	\begin{align*}
	\mathrm{i)}\,\,&\text{$f(0) = 0$ and $f(z) > 0$ for all $z > 0$}; \\
	\mathrm{ii)}\,\,& \exists\lim_{z\to+\infty}\frac{f(z)}{z} = +\infty\\
	\mathrm{iii)}\,\,&\int_1^{+\infty}\frac{1}{f(z)}\,dz < +\infty	
	\end{align*}
\end{itemize}
Clearly, assumption $(H)_f$ is satisfied when $f(u)=u^p$ with 
$p>1$.


\medskip

The study of finite-time blow-up and global existence of solutions to
the problem \eqref{eq:mainPbParabolicIntro} with $\mathbf{a}=1$ and $\mathbf{b}=0$ has been widely developed in the literature, specially whenever $f(u)=u^p, p>1$.
We refer to \cite{BPV} for a detailed account of the various available
results. Here, we restrict ourselves to recalling the early results
in \cite{F, H, KST}, as well as the survey papers
 \cite{BB, DL, deP, Levine, Mitidieri2, QS}, which also contain analogous
results for nonlinear problems.  On the other hand, the case $\mathbf{a}=0$ and $\mathbf{b}=1$ was treated
in \cite{Sug, LaisterS}.  Moreover, for $p>1+\frac{2s}{N}$, global-in-time solutions were
investigated in \cite{IKK} (see also \cite{HKN}), where the
asymptotic behavior of solutions as $t\to+\infty$ was also studied.

In recent years, mixed operators of the form $\mathcal L$ have
emerged as a significant object of study in the analysis of partial
differential equations. A growing literature has been devoted to the
qualitative properties of solutions to equations driven by such
operators, especially in the elliptic setting, but also in the
parabolic one; see, for instance, \cite{Biagi2}-\cite{BonfII}.

Part of this interest stems from the probabilistic interpretation of
these operators. Indeed, they naturally arise in connection with
stochastic dynamics combining different mechanisms, such as standard
diffusion and jump processes of L\'evy type. In addition, mixed
operators have proved useful in the modeling of several phenomena from
the applied sciences. A notable example is provided by models for
optimal animal foraging strategies; see, e.g., \cite{DPLV23, DV21} and the
references therein.

As for problem \eqref{eq:mainPbParabolicIntro} in the general case of a mixed local--nonlocal
operator with $f(u)=u^p$, it was shown in \cite[Theorem 3.3]{BPV} (see also \cite{DPF, KT}) that,
if $p \le 1+\frac{2s}{N}$, then every solution corresponding to any
nontrivial initial datum $u_0 \not\equiv 0$ blows up in finite time.
On the other hand, if $p > 1+\frac{2s}{N}$ and $u_0$ is sufficiently
small, then there exists a global-in-time solution.

We now focus on the case in which $p>1$ is arbitrary, while the
initial datum $u_0$ is sufficiently large. It is well known that, in
the case $\mathbf{a}=1$ and $\mathbf{b}=0$ (see, e.g., \cite{BL,deP}), if $u_0$ is
big enough, then the corresponding solution blows up in finite for every $p>1$.

Some related results in this direction were obtained
in \cite[Theorem 1.2]{DPF} for problem \eqref{eq:mainPbParabolicIntro} posed in a bounded domain,
by different methods relying heavily on the assumption
that $\Omega$ is bounded.

\medskip

In this paper, we show that, for problem
\eqref{eq:mainPbParabolicIntro}, finite-time blow-up occurs provided
that the initial datum is sufficiently large. Note that the assumption on the initial datum will be stated explicitly. In particular, when
$f(u)=u^p$, our result applies to every $p>1$. We emphasize that, to the best of our knowledge, the result
established here is new even in the special case $\mathbf{a}=0$, namely for
the fractional operator.

To obtain our result, we adapt to the present setting the
so-called Kaplan method (see \cite{Kaplan}), which was used, for instance,
in \cite{BL,deP} in the case $\mathbf{a}=1$, $\mathbf{b}=0$, and in \cite{PZ1} for the
semilinear heat equation posed on an infinite combinatorial graph.

We first derive a general condition ensuring finite-time blow-up of
the solution whenever $u_0$ is sufficiently large; see
Theorem \ref{thm:MainNonExistenceKaplan}. The proof of this theorem
requires particular care because of the nonlocal part of the
operator. The result relies on the existence of a special function $\kappa$,
which we call a {\it Kaplan function}. Besides satisfying suitable
regularity and integrability properties, $\kappa$  is a positive
subsolution of equation
\[
-\mathcal L \kappa + \lambda \kappa = 0
\qquad \text{pointwise in } \mathbb{R}^N,
\]
for some $\lambda>0$.

Given a solution $u(t)$ of problem \eqref{eq:mainPbParabolicIntro}, we define the function
$\Phi \colon [0,+\infty)\to[0,+\infty)$ by
\[
\Phi(t):=\int_{\mathbb{R}^N}\kappa u(t)\,dx.
\]
We then show that
\[
\Phi'(t)+\lambda\Phi(t)\ge f(\Phi(t))
\qquad \text{for all } t>0.
\]
If the initial datum $u_0$ is sufficiently large, then this property
is inherited by $\Phi(0)$. A suitable analysis of the above ordinary
differential inequality with initial condition $\Phi(0)$ then yields finite-time blow-up.

A further crucial step is the explicit construction of the
subsolution $\kappa$ (see Theorem \ref{thm:ExplicitKaplan}), which is carried out ad hoc by treating the
fractional part of the operator with special care, since explicit
computations are less straightforward in the nonlocal setting. By combining the general criterion with the construction of the
function $\kappa$, we deduce our main result, namely
Theorem \ref{thm:NonexistenceExplicit}. In this connection, let us point out that, within our framework
$\mathbf{a}\geq 0$ and $\mathbf{b}>0$, the Kaplan function is of the form
\[
\kappa(x):=C\varepsilon^{\frac N2}
\bigl(1+\varepsilon |x|^2\bigr)^{-\beta},
\qquad x\in\mathbb R^N,
\]
for some constants $C>0$, $\varepsilon>0$ and $\beta>0$.

By contrast, in the case $\mathbf{a}=1$ and $\mathbf{b}=0$,
which does not fall within the scope of our analysis, it is well known
(see \cite{BL,deP}) that the Kaplan function is given by
\[
\kappa(x):=Ck^{\frac N2}e^{-k|x|^2},
\qquad x\in\mathbb R^N,
\]
for some constants $C>0$ and $k>0$.

This difference in the form of $\kappa$ is due to the presence or
absence of the fractional part in the operator $\mathcal L$.

In addition, using Theorem \ref{thm:NonexistenceExplicit}, in which $u_0$ is assumed to be
sufficiently large and $p>1$, we deduce a blow-up result under the
assumptions that $u_0$ is arbitrary, as long as it is nontrivial, and
that $f(u)=u^p$ with $p<1+\frac{2s}{N}$. In this way, we recover a result that, as
mentioned above, had already been established in \cite{BPV}.

\medskip
\bigskip

\noindent -\,\,\emph{Plan of the paper}. In Section \ref{sec.Prel}, we fix the notation, recall the main properties of the
operator $\mathcal L$ and of the associated heat kernel, and give the
definition of solution.

In Section \ref{sec:MainThm}, we state and prove the main results. We first introduce
the notion of Kaplan function and establish a general criterion for
the nonexistence of global solutions; see
Theorem \ref{thm:MainNonExistenceKaplan}. We next construct the Kaplan function, as stated in
Theorem \ref{thm:ExplicitKaplan}, and then use it to state and prove
the main theorem, namely
Theorem \ref{thm:MainNonExistenceKaplan}. Finally, we discuss (see Theorem \ref{teo3}) the application of the
previous results to the special case $f(u)=u^p$ with
$p<1+\frac{2s}{N}$, for every nontrivial initial datum.

\bigskip
\medskip

\noindent-\,\,\textbf{Conflict of interest.} The author states no conflict of interest.

\noindent-\,\,\textbf{Data availability statement.} There are no data associated with this research.

\section{Relevant notation and preliminaries}\label{sec.Prel}

In this section we fix the notation that will be used throughout the paper and introduce the main definitions related to the mixed operator
$$\mathcal{L} = -\mathbf{a}\Delta + \mathbf{b}(-\Delta)^s$$
(and the associated Cauchy problem \eqref{eq:mainPbParabolicIntro}).

\subsection{Notation.}

Throughout the paper we shall use the notation listed below; we therefore refer the reader to this list whenever non-standard notation is encountered.

\begin{itemize}
 \item We denote by $\R^+$ (resp.\,$\R^+_0$) the interval $(0,+\infty)$ (resp.\,$[0,+\infty)$).
 \vspace*{0.1cm}

 \item Given $x_0\in\R^N$ and $r > 0$, we denote by $B_r(x_0)$ the open Euclidean ball centered at $x_0$ with radius $r$; in the special case $x_0 = 0$, we simply write $B_r$.
 \vspace*{0.1cm}

 \item Given two vectors $v_1,v_2 \in \mathbb{R}^N$, we denote by $\langle v_1,v_2\rangle$ their scalar product.
  \vspace*{0.1cm}
	
 \item Given $0<T\leq+\infty$, we denote by $S_T$ the (possibly infinite) strip $\R^N\times (0,T)$; in the case $T = +\infty$, we simply write $S$ in place of $S_{+\infty}$.
 \vspace*{0.1cm}

 \item If $A$ is a subset of some Euclidean space $\R^m$ (with $m\geq 1$), we denote by $\mathbf{1}_A$ the indicator function of $A$, namely
 $$\mathbf{1}_A(z) =
 \begin{cases}
 1 & \text{if $z\in A$}, \\
 0 & \text{if $z\notin A$}.
 \end{cases}
 $$
 \vspace*{0.1cm}

 \item Given $s\in (0,1)$, we denote by $\mathbb{L}^s(\mathbb{R}^N)$ the \emph{tail space}
 \begin{equation} \label{eq:spaceL1s}
 \mathbb{L}^s(\R^N) :=
 \Big\{u:\R^N\to\R:\ \|u\|_{1,s} := \int_{\R^N}\frac{|u(x)|}{1+|x|^{N+2s}}\,dx<+\infty\Big\}.
 \end{equation}

 \item Let $A$ be a subset of some Euclidean space $\R^m$ (with $m\geq 1$). We set
 \begin{equation} \label{eq:spaceCpositive}
 C_+(A) = \big\{f:A\to\mathbb{R}:\ f\in C(A)\ \text{and}\ f\geq 0\ \text{on } A\big\}.	
 \end{equation}

 \item Let $\Omega\subseteq\mathbb{R}^N$ be an open set and $I\subseteq\mathbb{R}$ an open interval. We set
 \begin{equation} \label{eq:C21xtDef}
 C^{2,1}_{x,t}(\Omega\times I) = \Big\{u:\Omega\times I\to\mathbb{R}:\
 \begin{array}{l}
 u(\cdot,t)\in C^2(\Omega)\ \text{for every } t\in I,\\
 u(x,\cdot)\in C^1(I)\ \text{for every } x\in\Omega
 \end{array}
 \Big\}.
 \end{equation}

 \item We denote by $\Gamma$ the Euler Gamma function, that is, the unique analytic extension to open set $\mathcal{D} = \mathbb{C}\setminus\{k\in\mathbb{Z}:\,k\leq 0\}$ of the map
 $$\{z\in\mathbb{C}:\ \mathrm{Re}(z)>0\}\ni z \mapsto \int_0^{+\infty}t^{z-1}e^{-t}\,dt.$$
\end{itemize}

\subsection{Pointwise theory for the Cauchy problem \eqref{eq:mainPbParabolicIntro}}
Having fixed the notation, we briefly discuss the \emph{classical theory} for the Cauchy problem \eqref{eq:mainPbParabolicIntro} associated with $\mathcal{L}$.
\medskip

\noindent\textbf{1) The fractional Laplacian.}
Let $s\in (0,1)$ be fixed and let $u:\R^N\to\R$. The \emph{fractional Laplacian} of order $s$ of $u$ at a point $x\in\R^N$ is defined by
\begin{equation} \label{eq:defDeltas}
\begin{split}
 (-\Delta)^s u(x) & = C_{N,s}\,\mathrm{P.V.}\int_{\R^N}\frac{u(x)-u(y)}{|x-y|^{N+2s}}\,dy \\
& = C_{N,s}\,\lim_{\varepsilon\to 0^+}\int_{\{|x-y|\geq\varepsilon\}}\frac{u(x)-u(y)}{|x-y|^{N+2s}}\,dy,
\end{split}
\end{equation}
whenever the limit exists and is finite.

Here $C_{N,s} > 0$ is a normalization constant, explicitly given by
$$C_{N,s} = \frac{2^{2s}\Gamma((N+2s)/2)}{\pi^{N/2}|\Gamma(-s)|}.$$
In order for $(-\Delta)^s u(x)$ to be well defined, suitable growth conditions on $u$ are required, both as $y\to\infty$ and as $y\to x$. In this direction we recall the following result (see \cite{KKP,Silv0}).

\begin{proposition} \label{prop:welldefDeltas}
Let $\varnothing\neq\Omega\subseteq\R^N$ be an open set and let $\mathbb{L}^s(\mathbb{R}^N)$ be defined as in \eqref{eq:spaceL1s}.

Then, the following assertions hold.
\begin{itemize}
 \item[{i)}] If $0<s<1/2$ and $u\in C_{\mathrm{loc}}^{2s+\gamma}(\Omega)\cap \mathbb{L}^s(\R^N)$ for some $\gamma \in (0,1-2s)$, then
 $$
 (-\Delta)^s u(x) = C_{N,s}\int_{\R^N}\frac{u(x)-u(y)}{|x-y|^{N+2s}}\,dy\qquad\forall\,\,x\in\Omega.
 $$

 \item[{ii)}] If $1/2<s<1$ and $u\in C^{1,2s-1+\gamma}_{\mathrm{loc}}(\Omega)\cap \mathbb{L}^s(\R^N)$ for some $\gamma\in (0,2-2s)$, then
 $$
 (-\Delta)^s u(x) = -\frac{C_{N,s}}{2}\int_{\R^N}\frac{u(x+z)+u(x-z)-2u(x)}{|z|^{N+2s}}\,dz\qquad\forall\,\,x\in\Omega.
 $$

\end{itemize}
Moreover, in both cases $(-\Delta)^s u\in C(\Omega)$.
\end{proposition}

\begin{remark} \label{rem:BoundedLs}
Observe that the weight $\kappa(x) = (1+|x|^{N+2s})^{-1}$ belongs to $L^r(\mathbb{R}^N)$ for every $1\leq r\leq +\infty$. Consequently,
\begin{equation} \label{eq:LinfinLs}
L^r(\mathbb{R}^N)\subseteq \mathbb{L}^s(\mathbb{R}^N)\quad\forall\,\,1\leq r\leq+\infty.
\end{equation}
This fact will be used repeatedly in the sequel.
\end{remark}

In the special case $\Omega = \R^N$ and $u\in\mathcal{S}$ (the Schwartz space of rapidly decreasing functions), an equivalent definition of $(-\Delta)^s u$ can be given via the Fourier transform $\mathfrak{F}$.

\begin{proposition} \label{prop:DeltasFourier}
Let $u\in\mathcal{S}\subseteq \mathbb{L}^s(\R^N)$. Then
\begin{equation} \label{eq23}
(-\Delta)^s u(x) = \mathfrak F^{-1} \big( |\xi|^{2s}\mathfrak F u  \big)(x)
\quad \text{for all } x\in\R^N.
\end{equation}
\end{proposition}

\begin{remark} \label{rem:hintweakLap}
Since the Fourier transform extends to an isometry on $L^2(\R^N)$, identity \eqref{eq23} allows one to define the fractional Laplacian for less regular functions.

More precisely, one can set
$$(-\Delta)^s u = \mathfrak F^{-1} \big( |\xi|^{2s}\mathfrak F u  \big)\in L^2(\R^N)$$
for every $u\in L^2(\R^N)$ such that $|\xi|^{2s}\mathfrak F u \in L^2(\R^N)$.
\end{remark}
Due to its relevance in the sequel,
we  state
in the next theorem the well-known frac\-tional version of the \emph{Leibniz formula} and of the \emph{integration-by-parts formula}
for the Laplacian.
\begin{theorem}  \label{thm:LeibnizandByParts}
The following assertions hold.
\begin{enumerate}
	\item Assume that $f,g\in\mathbb{L}^s(\R^N)$ satisfy
  the following properties:
\begin{align*}
	\mathrm{a)}\,\,&\text{$f,g\in C_{\mathrm{loc}}^{2s+\gamma}(\R^N)$ for some $\gamma\in (0,1-2s)$ if $0<s<1/2$, \textbf{or}} \\
	&\text{$f,g\in C^{1,2s+\gamma-1}_{\mathrm{loc}}(\R^N)$
	for some
  	$\gamma\in (0,2-2s)$ if $1/2<s<1$}; \\[0.15cm]
  	\mathrm{b)}\,\,&fg\in\mathbb{L}^s(\R^N);
\end{align*}
 Then, the following \emph{Leibniz formula} holds:
 \begin{equation}\label{eq:Leibniz_rule}
	(-\Delta)^s[fg](x) = f(x) (-\Delta)^s g(x) + (-\Delta)^sf(x) g(x) - \mathcal{B}(f,g)(x),
\end{equation}
\noindent where the bilinear form $\mathcal{B}$ is defined as follows:
\begin{equation}\label{eq:def_B}
	\mathcal{B}(f,g)(x):=C_{N,s}\int_{\mathbb{R}^{N}}\dfrac{(f(x)-f(y))(g(x)-g(y))}{|x-y|^{N+2s}}\, dy \quad \textrm{for all } x \in \mathbb{R}^{N}.
\end{equation}

\item Let $\Omega\subseteq\R^N$ be a bounded open set,
and let $u,v$ satisfy the following properties:
	\begin{align*}
	\mathrm{a)}\,\,&\text{$u,v\in C_{\mathrm{loc}}^{2s+\gamma}(\Omega)$ for some $\gamma\in (0,1-2s)$ if $0<s<1/2$, \textbf{or}} \\
	&\text{$u,v\in C^{1,2s+\gamma-1}_{\mathrm{loc}}
	(\Omega)$ for some
  	$\gamma\in (0,2-2s)$ if $1/2<s<1$}; \\[0.15cm]
  	\mathrm{b)}\,\,&\text{$u\in L^\infty(\R^N)$
  	and $v\in C^s(\R^N)$}; \\[0.15cm]
  	\mathrm{c)}\,\,&\text{$v\equiv 0$ on $\R^N\setminus\Omega$ and $(-\Delta)^sv\in L^1(\R^N)$}.
\end{align*}
Then, the following formula holds:
\begin{equation}\label{eq:Integration_by_parts}
	\int_\Omega v(-\Delta)^s u\,dx = \int_{\mathbb{R}^{N}}u (-\Delta)^s v \, dx
	\end{equation}
\end{enumerate}
\end{theorem}
\begin{remark} \label{rem:HowtouseFormula}
In the proof of our main theorem, the above formulas \eqref{eq:Leibniz_rule}
and \eqref{eq:Integration_by_parts} will be applied under the following assumptions:
\begin{enumerate}
\item $f\in C_c^\infty(\R^N)$ and $g\in C^2(\R^N)\cap L^1(\R^N)$
(for \eqref{eq:Leibniz_rule});
\vspace{0.1cm}

\item  $v=fg\in C_c^2(\R^N)$
(for \eqref{eq:Integration_by_parts}).
\end{enumerate}
Under these assumptions, the hypotheses of Theorem
\ref{thm:LeibnizandByParts} are satisfied; in particular,
since (in our context) $v\in C^2_0(\R^N)$, one has
$(-\Delta)^s v\in L^1(\R^N)$.
since
\end{remark}

\noindent \textbf{2) Classical solutions to problem \eqref{eq:mainPbParabolicIntro}.}
Having reviewed the basic properties of the fractional Laplacian, we can now introduce the notion of \emph{classical solution} to the Cauchy problem \eqref{eq:mainPbParabolicIntro}. From now on, we tacitly assume that
$$\mathbf{a} \geq 0,\qquad \mathbf{b}>0,\qquad \text{and}\qquad \text{$f$ satisfies assumption $(H)_f$}.$$
We stress that the choice $\mathbf{a}=0$,
which corresponds to the purely nonlocal operator
$$\LL=\mathbf{b}(-\Delta)^s,$$
is fully admissible; therefore, all the results that follow remain valid in this case as well. However, since the main goal of this paper is to investigate the interplay between the local and the nonlocal parts of the operator, we formulate the theory having primarily in mind the genuinely mixed case $\mathbf{a},\mathbf{b}>0$. In particular, the definition of classical solution to \eqref{eq:mainPbParabolicIntro} is tailored to this setting (see Remark \ref{rem:DefWellPosed}-(3)).

\begin{definition}\label{defsole}
Let $u_0\in C_+(\mathbb{R}^N)\cap L^\infty(\mathbb{R}^N)$. A function
$$u: \mathbb{R}^N\times[0,T)\to \mathbb{R}$$
(for some $0<T\leq+\infty$) is called a \emph{classical solution} to \eqref{eq:mainPbParabolicIntro} if:
\begin{enumerate}
 	\item $u\in C_+(\mathbb{R}^N\times[0,T))\cap  C^{2,1}_{x,t}(S_T)$;	
 	\vspace{0.1cm}
 	
	\item $u\in L^\infty(\mathbb{R}^N\times[0,\tau])$ for every $0<\tau<T$;

 	\vspace{0.1cm}
 	
	\item $u$ satisfies \eqref{eq:mainPbParabolicIntro} pointwise, namely
	\begin{align*}
	\mathrm{a)}\,\,&u_t(x,t)+\mathcal{L}[u(\cdot,t)](x) = f(u(x,t))	\quad\forall (x,t)\in S_T, \\
	\mathrm{b)}\,\,&u(x,0) = u_0(x)\quad\forall x\in\mathbb{R}^N.
	\end{align*}
\end{enumerate}
If $T=+\infty$, we say that $u$ is a \emph{global solution} (to \eqref{eq:mainPbParabolicIntro}).
\end{definition}

\begin{remark} \label{rem:DefWellPosed}
We make a few remarks concerning Definition \ref{defsole}.

\begin{enumerate}
	\item First, we point out that Definition \ref{defsole} is \emph{well posed} in the following sense: if
 $$u:\mathbb{R}^N\times [0,T)\to\mathbb{R}$$
 (with $0<T\leq+\infty$) satisfies properties (1)--(2), then the quantity
	$\LL[u(\cdot,t)](x)$
	is well defined for every $x\in\mathbb{R}^N$ and every $t\in (0,T)$.
	\vspace{0.1cm}

	Indeed, fixing $t\in (0,T)$, properties (1)--(2) yield
	$$0\leq u(x,t)\leq \|u\|_{L^\infty(\mathbb{R}^N\times[0,t])}
	\qquad \forall\,\,x\in\mathbb{R}^N,$$
	so that $u(\cdot,t)\in L^\infty(\mathbb{R}^N)\subseteq \mathbb{L}^s(\mathbb{R}^N)$ (see \eqref{eq:LinfinLs}) and
	\begin{equation} \label{eq:ucdottBounded}
	\|u(\cdot,t)\|_{L^\infty(\mathbb{R}^N)} \leq \|u\|_{L^\infty(\mathbb{R}^N\times[0,t])}<+\infty.
	\end{equation}	
	As a consequence, since we also have $u(\cdot,t)\in C^2(\mathbb{R}^N)$ by property (1),  from Proposition \ref{prop:welldefDeltas} we derive that $\LL[u(\cdot,t)]$ is well defined pointwise in $\mathbb{R}^N$, and
	$$x\mapsto \LL[u(\cdot,t)](x)\in C(\mathbb{R}^N).$$

	\item Let $u:\mathbb{R}^N\times[0,+\infty)\to \mathbb{R}^+_0$ be a \emph{global solution} to the Cauchy problem \eqref{eq:mainPbParabolicIntro}. Then, by \eqref{eq:ucdottBounded} (and the arbitrariness of $t\in(0,+\infty)$), we have
	$$u(\cdot,t)\in L^\infty(\mathbb{R}^N)\quad \text{for every } t>0.$$
	
	\item As already discussed, the above definition is \emph{modeled on} the genuinely mixed case $\mathbf{a},\mathbf{b}>0$, since $C^2$-regularity in the space variable is not needed when $\mathbf{a}=0$. In that case, taking Proposition \ref{prop:welldefDeltas} into account, it is enough to require
\begin{itemize}
	\item 
	$u(\cdot,t)\in C^{2s+\gamma}_{\mathbf{loc}}(\R^n)$
	for some $\gamma\in (0,1-2s)$ if $0<s<1/2$;
	\vspace{0.2cm}
	
	\item $u(\cdot,t)\in C^{1,2s-1+\gamma}_{\mathbf{loc}}(\R^n)$
	for some $\gamma\in (0,2-2s)$ if $1/2<s<1$.
	\end{itemize}
\end{enumerate}
\end{remark}

\section{Kaplan functions for $\LL$ and the main results} \label{sec:MainThm}

Thanks to Definition \ref{defsole}, we are now in a position to state and prove our main non-exi\-stence results for global solutions to the Cauchy problem \eqref{eq:mainPbParabolicIntro}.
Before doing so, and in order to improve the readability of the statement, we introduce the following auxiliary \emph{ad hoc} definition.
Throughout what follows, given $\beta > 0$, we denote by $\Psi_\beta$ the weight function
$$\Psi_\beta(x) = \frac{1}{(1+|x|^2)^\beta}.$$
Furthermore, for every $\lambda > 0$ we define
\begin{equation} \label{eq:defsflambda}
\mathbf{s}_f(\lambda) = \inf\Big\{z > 0:\,\frac{f(z)}{z} > \lambda\Big\}\in[0,+\infty).	
\end{equation}
We note that, for every $\lambda > 0$, assumption $(H)_f$-(ii) ensures that
$$
\mathcal{S}(\lambda) = \Big\{z > 0:\,\frac{f(z)}{z} > \lambda\Big\}\neq \varnothing,
$$
and therefore $\mathbf{s}_f(\lambda)$ is well defined. Moreover, in the prototypical case $f(z)=z^p$ (with $p>1$), it is straightforward to check that
\begin{equation} \label{eq:sflambdapower}
\mathbf{s}_f(\lambda) = \lambda^{1/(p-1)}.
\end{equation}
\begin{definition}[Kaplan function for $\LL$] \label{def:KaplanFunction}
We say that a function $\kappa:\mathbb{R}^N\to\mathbb{R}$ is a \emph{Kaplan function} for the operator $\LL$ if it satisfies the following properties:
\begin{enumerate}
	\item $\kappa\in C^2(\mathbb{R}^N)$ and $\kappa > 0$ in $\mathbb{R}^N$;
	\vspace{0.1cm}
		 		
	\item $\kappa,\,|\nabla\kappa|\in L^1(\mathbb{R}^N)$ and $\|\kappa\|_{L^1(\mathbb{R}^N)} = 1$;
	\vspace{0.1cm}
	
	\item there exists $\lambda > 0$ such that
	$$-\LL\kappa+\lambda\kappa \geq 0 \quad \text{pointwise in }\mathbb{R}^N.$$
\end{enumerate}
\end{definition}

\begin{remark} \label{re:WellPosedDefKaplan}
Arguing as in Remark \ref{rem:DefWellPosed}, we observe that Definition \ref{def:KaplanFunction} is well posed in the following sense: if $\kappa:\mathbb{R}^N\to\mathbb{R}$ satisfies properties (1)--(2), then $\LL\kappa$ can be computed pointwise in $\mathbb{R}^N$.
This follows directly from Proposition \ref{prop:welldefDeltas}, taking into account that
$$L^1(\mathbb{R}^N)\subseteq\mathbb{L}^s(\mathbb{R}^N).$$
\end{remark}

In what follows, we will need to exhibit a suitable cut-off function built upon a smooth function $\chi:[0,+\infty)\to [0,1]$ satisfying the following properties:
\begin{itemize}
	\item[i)]  $\chi \in C^{\infty}_{0}([0,+\infty)$;
	\item[ii)] $\chi \equiv 1$ in $[0,1]$ and $\chi \equiv 0$ in $[2,+\infty)$;
	\item[iii)] there exist two positive constants $C_1,C_2>0$ such that
	$$|\chi'(t)| \leq C_1 \,  \mathbf{1}_{(1,2)}(t), \quad |\chi''(t)|\leq C_{2} \, \mathbf{1}_{(1,2)}(t) \quad \textrm{for all } t \in [0,+\infty).$$
\end{itemize}
Once $\chi$ is as above, for any $R>0$ we define the function $\chi_{R}:\mathbb{R}^N \to [0,1]$ as
\begin{equation}\label{eq:Def_Chi_R}
	\chi_{R}(x):= \chi \left(\dfrac{|x|}{R}\right).
\end{equation}	

\begin{lemma}\label{lem:chiR}
	Let $\chi_{R}$ be defined as in \eqref{eq:Def_Chi_R}. Then:
	\begin{itemize}
		\item[a)] $\chi_{R}\in C^{\infty}_{0}(\mathbb{R}^{N})$, with $\chi_{R}\equiv 1$ in $\overline{B}_{R}$ and $\chi_{R}\equiv 0$ in $\mathbb{R}^{N} \setminus B_{2R}$;
		\item[b)] for all $x\in \mathbb{R}^{N}$, $\lim_{R \to +\infty}\chi_{R}(x)=1$;
		\item[c)] there exists a positive constant $C_3>0$ such that
		\begin{equation}\label{eq:stima_grad_chiR}
			|\nabla \chi_{R}(x)| \leq \dfrac{C_3}{R} \, \mathbf{1}_{B_{2R}\setminus \overline{B}_{R}}(x) \quad \textrm{for all } x \in \mathbb{R}^{N};
		\end{equation}	
		\item[d)] there exists a positive constant $C_4>0$ such that
		\begin{equation}
			|\Delta \chi_R(x)| \leq \dfrac{C_4}{R^2} \, \mathbf{1}_{B_{2R}\setminus \overline{B}_{R}}(x) \quad \textrm{for all } x \in \mathbb{R}^{N};
		\end{equation}
	\end{itemize}
\end{lemma}

\begin{theorem} \label{thm:MainNonExistenceKaplan}
Assume that there exists a Kaplan function $\kappa$ for $\LL$, which additionally sati\-sfies the following property: there exists a constant $c > 0$ such that
\begin{equation} \label{eq:ExtraKaplanIntbyParts}
	0\leq \kappa+|\nabla \kappa|\leq c\,\Psi_\beta \quad \text{pointwise in }\mathbb{R}^N;
\end{equation}
If $u_0\in C_+(\mathbb{R}^N)\cap L^\infty(\mathbb{R}^N)$ satisfies
\begin{equation} \label{eq:KaplanConditionuzero}
	\int_{\mathbb{R}^N}\kappa(x)u_0(x)\,dx > \mathbf{s}_f(\lambda)
\end{equation}
\emph{(}where $\lambda > 0$ is as in Definition \ref{def:KaplanFunction}-(3), and $\mathbf{s}_f(\lambda)$ is as in \eqref{eq:defsflambda}\emph{)},  the Cauchy problem \eqref{eq:mainPbParabolicIntro} with initial datum $u_0$ does \emph{not} admit global classical solutions which satisfy
   \begin{equation}\label{e2}
   \kappa u_t(\cdot,t) \in L^{1}(\mathbb{R}^N)\quad\forall\,\,t\in (0,T).
   \end{equation}
   
\begin{proof}
We will argue by contradiction therefore assuming that
$u$ is a global solution to \eqref{eq:mainPbParabolicIntro}.
Let us now fix $t>0$ and $R>0$. Multiplying first equation \eqref{eq:mainPbParabolicIntro} by $\chi_{R}\kappa$ and then integrating over the entire space $\mathbb{R}^N$, gives
\begin{equation}\label{eq:testando_con_chiRk}
	\begin{aligned}
	\int_{\mathbb{R}^{N}}\chi_{R}\kappa u_{t}\, dx &= \int_{\mathbb{R}^{N}}\chi_{R}\kappa \, (-\mathcal{L}u)\, dx + \int_{\mathbb{R}^{N}}\chi_{R}\kappa f(u)\, dx\\
	&= \mathbf{a}\, \int_{\mathbb{R}^{N}}\chi_{R}\kappa \, (\Delta u)\, dx - \mathbf{b}\, \int_{\mathbb{R}^{N}}\chi_{R}\kappa \,(-\Delta)^{s} u\, dx + \int_{\mathbb{R}^{N}}\chi_{R}\kappa f(u)\, dx\\
	&=: \mathbf{a} \, \mathcal{I}_1 - \mathbf{b} \, \mathcal{I}_2 + \mathcal{I}_3.
\end{aligned}
\end{equation} 	
Let us start with $\mathcal{I}_1$. Recalling Lemma \ref{lem:chiR}-a), there exists a compact set $D$ such that $\mathrm{supp}(\chi_{R})\subseteq \overline{B}_{2R}\subset D$ and with the additional properties that $\chi_{R}\equiv 0$ on $\partial D$ and in $\mathbb{R}^{N}\setminus D$. Therefore, integrating by parts we get
\begin{equation}\label{eq:I1}
	\begin{aligned}
		\mathcal{I}_1 &= \int_{D}\chi_{R}\kappa \, (\Delta u)\, dx = -\int_{D}\langle \nabla(\chi_R k), \nabla u\rangle \, dx =  \\
		&= -\int_{D}\chi_R \langle \nabla k, \nabla u\rangle \, dx -\int_{D}k\langle \nabla \chi_R, \nabla u\rangle \, dx\\
		&=: -\overline{\mathcal{I}}_1 - \widetilde{\mathcal{I}}_1.
	\end{aligned}
\end{equation}
Now, focusing on $\overline{\mathcal{I}}_1$, we notice that
\begin{equation}\label{eq:overI1}
	\begin{aligned}
		\overline{\mathcal{I}}_1 &= \int_{D} \langle \nabla k, \nabla (\chi_{R}u)\rangle \, dx - \int_{D}u\langle \nabla k, \nabla \chi_{R}\rangle \, dx \\
		& = -\int_{D}\chi_{R}u \Delta k \, dx - - \int_{D}u\langle \nabla k, \nabla \chi_{R}\rangle \, dx \\
		&= -\int_{\mathbb{R}^{N}}\chi_{R}u \Delta k \, dx -  \int_{\mathbb{R}^{N}}u\langle \nabla k, \nabla \chi_{R}\rangle \, dx,
	\end{aligned}
\end{equation}
\noindent while, as long as $\widetilde{\mathcal{I}}_{1}$ is concerned
\begin{equation}\label{eq:tildeI1}
	\begin{aligned}
		\widetilde{\mathcal{I}}_{1} &= \int_{D}\langle \nabla \chi_R, \nabla (ku)\rangle \, dx - \int_{D}u \langle \nabla \chi_{R},\nabla k\rangle \, dx\\
		&= - \int_{D}ku \nabla \chi_{R} \, dx - \int_{D}u \langle \nabla \chi_{R},\nabla k\rangle \, dx\\
		&= - \int_{\mathbb{R}^{N}}ku \nabla \chi_{R} \, dx - \int_{\mathbb{R}^{N}}u \langle \nabla \chi_{R},\nabla k\rangle \, dx.
	\end{aligned}
\end{equation}
Combining \eqref{eq:I1}, \eqref{eq:overI1} and \eqref{eq:tildeI1}, we finally get that
\begin{equation*}
	\begin{aligned}
	\mathcal{I}_{1} &= \int_{\mathbb{R}^{N}}\chi_{R}u\Delta k \, dx + \int_{\mathbb{R}^N}k u \Delta \chi_{R} \, dx + 2\int_{\mathbb{R}^{N}}u \langle \nabla k, \nabla \chi_{R}\rangle \, dx\\
	&=:\mathcal{I}'_{1} + \mathcal{I}''_{1} +  \mathcal{I}'''_{1}.
	\end{aligned}
\end{equation*}
Let us now deal with the nonlocal part provided by $\mathcal{I}_2$. By exploiting \eqref{eq:Integration_by_parts} and \eqref{eq:Leibniz_rule} (applied here with $u = u(\cdot,t)$ and $v = fg = \chi_R\kappa$, see Remark
\ref{rem:HowtouseFormula}),  we get
\begin{equation*}
	\begin{aligned}
		\mathcal{I}_{2} &= \int_{\mathbb{R}^{N}}u (-\Delta)^s(\chi_{R}k)\, dx \\
		&= \int_{\mathbb{R}^{N}}u \chi_R (-\Delta)^{s}k \, dx + \int_{\mathbb{R}^{N}}uk (-\Delta)^s \chi_{R}\, dx + \int_{\mathbb{R}^{N}}u\mathcal{B}(\chi_{R},k)\, dx\\
		&=: \mathcal{I}'_{2} + \mathcal{I}''_{2}+ \mathcal{I}'''_{3}.
	\end{aligned}
\end{equation*}
Now, by \cite[Lemma 3.2]{BMP}, it follows that
\begin{equation}\label{eq:Limite_locale}
	\int_{\mathbb{R}^N}k u \Delta \chi_{R} \, dx + 2\int_{\mathbb{R}^{N}}u \langle \nabla k, \nabla \chi_{R}\rangle \, dx \to 0 \quad \textrm{ as } R \to +\infty,
\end{equation}
\noindent and
\begin{equation}\label{eq:Limite_nonlocal}
	\int_{\mathbb{R}^{N}}uk (-\Delta)^s \chi_{R}\, dx + \int_{\mathbb{R}^{N}}u\mathcal{B}(\chi_{R},k)\, dx \to 0 \quad \textrm{as } R \to +\infty.
\end{equation}
Moreover, recalling Lemma \ref{lem:chiR}-b), and using the Dominated Conveergence Theorem, it follows that
\begin{equation}\label{eq:Limite_fu}
	\mathcal{I}_{3} = \int_{\mathbb{R}^N}\chi_R k f(u)\, dx \to \int_{\mathbb{R}^{N}}k f(u) \, dx \quad \textrm{ as } R \to +\infty.
\end{equation}
We are left with a couple of terms, namely $\mathcal{I}'_{1}$ and $\mathcal{I}'_{2}$. Here, we exploit first that $k$ is a Kaplan function for $\mathcal{L}$ and hence
\begin{equation}\label{eq:usando_Kaplan}
	\mathbf{a}\mathcal{I}'_{1} - \mathbf{b}\mathcal{I}'_{2} \geq -\lambda \int_{\mathbb{R}^{N}}\chi_R u k \, dx.
\end{equation}
Arguing as for the integral with $f(u)$, it follows that
\begin{equation}\label{eq:Limite_post_Kaplan}
	\int_{\mathbb{R}^{N}}\chi_R u k \, dx \to \int_{\mathbb{R}^{N}}u k \, dx \quad \textrm{ as } R \to +\infty,
\end{equation}
\noindent and that
\begin{equation}\label{eq:Limite_ut}
	\int_{\mathbb{R}^{N}}\chi_{R}\kappa u_{t}\, dx  \to \int_{\mathbb{R}^{N}}k u_t \, dx \quad \textrm{ as } R \to +\infty.
\end{equation}
Starting from \eqref{eq:testando_con_chiRk} and using \eqref{eq:usando_Kaplan},
we see that, thanks to \eqref{eq:Limite_ut}, \eqref{eq:Limite_locale}, \eqref{eq:Limite_nonlocal}, \eqref{eq:Limite_fu} and \eqref{eq:Limite_post_Kaplan}, we can pass to the limit as $R\to +\infty$ finding
\begin{equation}\label{eq:Dopo_il_Limite_in_R}
	\int_{\mathbb{R}^{N}}ku_t \, dx \geq -\lambda \int_{\mathbb{R}^{N}}k u \, dx + \int_{\mathbb{R}^{N}}k f(u)\, dx \quad \textrm{for all } t>0.
\end{equation}
Define now the function $\Phi:[0,+\infty)\to [0,+\infty)$
as
\begin{equation}\label{eq:Def_Phi}
	\Phi(t):= \int_{\mathbb{R}^{N}}k u(t) \, dx.
\end{equation}
Recalling that $u(\cdot,0)=u_0$ in $\mathbb{R}^N$ and \eqref{eq:KaplanConditionuzero}, it follows that
\begin{equation}\label{eq:Condizione_Iniziale_su_Phi}
	\Phi(0)> \mathbf{s}_f(\lambda).
\end{equation}
Due to the regularity of both $k$ and $u$, we have that $\Phi \in C^{1}((0,+\infty))$ and moreover
\begin{equation}\label{eq:Phi'}
	\Phi'(t) = \int_{\mathbb{R}^{N}}k u_t \, dx \quad \textrm{ for all } t>0.
\end{equation}
Since by Jensen's inequality
\begin{equation}\label{eq:usando_Jensen}
	\int_{\mathbb{R}^{N}}k f(u)\, dx \geq f\left(\int_{\mathbb{R}^{N}}k u \, dx \right),
\end{equation}
\noindent we can write \eqref{eq:Dopo_il_Limite_in_R} in terms of $\Phi$ (here using \eqref{eq:Phi'} as well) as a first order differential inequality of the form
\begin{equation}\label{e1}
	\Phi'(t) + \lambda \Phi(t) \geq f(\Phi(t))\quad \textrm{ for all } t >0,
\end{equation}
\noindent whose associated initial condition is given by \eqref{eq:Condizione_Iniziale_su_Phi}. By a direct qualitative analysis of \eqref{e1} and \eqref{eq:Condizione_Iniziale_su_Phi} (see e.g. \cite[Lemma 5.3]{PZ1}), we realize that $\Phi$ has to blow up in finite time therefore contradicting our initial assumption on the solution $u$. This closes the proof.
\end{proof}
\end{theorem}

\begin{remark}[On assumption \eqref{eq:ExtraKaplanIntbyParts}] \label{rem:OnEstimatekappa}
As will become clear in the proof of Theorem \ref{thm:MainNonExistenceKaplan}, assum\-ption \eqref{eq:ExtraKaplanIntbyParts} is purely technical, and is
related to the nonlocal nature of $\LL$: indeed, it allows us to control certain \emph{tail terms} arising from the nonlocal part of $\LL$ when applied to compactly supported functions. 

\vspace{0.1cm}

We also observe that, if $\beta > {N}/{2}$, then $\Psi_\beta\in L^1(\mathbb{R}^N)$, and hence \eqref{eq:ExtraKaplanIntbyParts} ensures that
$$\kappa,\,|\nabla\kappa|\in L^1(\mathbb{R}^N).$$
In particular, for such values of $\beta$, the integrability condition in Definition \ref{def:KaplanFunction}-(2) is automa\-tically satisfied.
\end{remark}

\begin{remark}[On assumption \eqref{e2}]\label{timeder}
As observed in Remark \ref{rem:OnEstimatekappa}, if
$\beta>\frac{N}{2}$, then $\kappa\in L^1(\mathbb{R}^N)$. Therefore,
if $u_t(t)\in L^\infty(\mathbb{R}^N)$ for every $t\in(0,T)$, then
\eqref{e2} is satisfied.

In this connection, let us note that, for purely local operators,
the above property of $u_t(t)$ is satisfied by the solutions
constructed via semigroup theory
(see, e.g., \cite[Propositions 7.1.10 and 7.3.1]{Lun}
and \cite[Remark1]{Pu2}). It is not clear whether an analogous
result remains valid in the present setting, and this issue lies
beyond the scope of the present paper.
\end{remark}

The above Theorem \ref{thm:MainNonExistenceKaplan} provides an abstract criterion for the non-existence of global solutions to the Cauchy problem \eqref{eq:mainPbParabolicIntro}.
To make this result effective, it is necessary to construct an explicit Kaplan function associated with the operator $\LL$.
\vspace{0.1cm}

This is the goal of the next theorem.
\begin{theorem} \label{thm:ExplicitKaplan}
Let $\beta > N/2$ and $\e\in(0,1]$ be fixed. Then, the function
$$\kappa_\e(x) = \frac{\e^{N/2}}{\mathbf{c}_\beta}\cdot\frac{1}{(1+\e|x|^2)^{\beta}}\qquad \big(\text{with $\mathbf{c}_\beta = \|\Psi_\beta\|_{L^1(\mathbb{R}^N)}<+\infty$}\big)$$
is a Kaplan function for $\LL$, further satisfying \eqref{eq:ExtraKaplanIntbyParts} in Theorem \ref{thm:MainNonExistenceKaplan}.

More precisely, there exists $\lambda_0  > 0$, independent of $\e$, such that
\begin{equation} \label{eq:Prop3strongerform}
\begin{gathered}
\text{$-\LL\kappa_\e+\lambda\kappa_\e\geq 0$ pointwise in $\R^N$}	 \\	
\text{for every $\lambda\geq \e^{1/s}\lambda_0$}
\end{gathered}	
\end{equation}
\end{theorem}
\begin{proof}
	We first prove the theorem when $\e = 1$. To begin with, we observe that
	$$\kappa_1(x) = \frac{1}{\mathbf{c}_\beta}\cdot\frac{1}{(1+|x|^2)^\beta} = \mathbf{c}_\beta^{-1}\Psi_\beta(x)$$
	is  smooth and strictly positive in $\mathbb{R}^N$ (hence,
	property (1) in Definition \ref{def:KaplanFunction} is fulfilled); more\-over, for every $x\in\mathbb{R}^N$ we have
	\begin{align*}
	|\nabla\kappa_1|(x) & = \frac{1}{\mathbf{c}_\beta}
	\cdot\frac{2\beta|x|}{(1+|x|^2)^{\beta+1}}
	\leq \frac{2\beta}{\mathbf{c}_\beta}\cdot\frac{1}{(1+|x|^2)^\beta} =
	\frac{2\beta}{\mathbf{c}_\beta}\Psi_\beta(x),
	\end{align*}
and thus $\kappa_1$ satisfies assumption \eqref{eq:ExtraKaplanIntbyParts} (with
$c = (2\beta+1)/\mathbf{c}_\beta$). In particular, since $\beta > N/2$, and by
the  definition of $\mathbf{c}_\beta$, we deduce that $\kappa_1$ satisfies also
property (2)
in Definition \ref{def:KaplanFunction}.

Hence, it remains to verify
 property (3).
To this end, we proceed by separately estimating the local part $\Delta\kappa_1$ and the nonlocal part $-(-\Delta)^s\kappa_1$ of $-\LL\kappa_1$.
\medskip

\noindent -\,\,\emph{Estimate of the local part.} For every $x\in\mathbb{R}^N$, we have
$$\Delta\kappa_1(x) = \frac{2\beta}{\mathbf{c}_\beta(1+|x|^2)^{\beta+2}}\big[(2\beta-N+2)|x|^2-N\big];$$
from this, recalling that $\beta > N/2$, we derive the following facts:
\begin{align}
\mathrm{i)}\,\,& \Delta\kappa_1(x)\geq -\frac{2\beta N}{\mathbf{c}_\beta(1+|x|^2)^{\beta+2}}\quad\forall\,\,x\in\mathbb{R}^N; \label{eq:Deltakappageq}\\[0.1cm]
 \mathrm{ii)}\,\,&\text{there exists $\mathcal{A} > 0$ such that $|\Delta\kappa_1|\leq \mathcal{A}$ pointwise on $\R^N$}.	\label{eq:DeltakappaBd}
\end{align}

\noindent-\,\,\emph{Estimate of the nonlocal part.}
We begin by recalling that, by the computations carried out in the proof of \cite[Corollary 4.1]{FerrVerb}, one can derive an explicit expression for $-(-\Delta)^s\kappa_1$. More precisely, there exists a constant $\vartheta > 0$ such that, for every $x\in\mathbb{R}^N$, it holds
\begin{align*}
-(-\Delta)^s\kappa_1(x)
&= -\frac{\vartheta}{\mathbf{c}_\beta}\cdot\mathcal{H}\big(N/2+s,\beta+s,N/2,-|x|^2\big) \\
&= -\frac{\vartheta}{\mathbf{c}_\beta(1+|x|^2)^{\beta+s}}\cdot
\mathcal{H}\Big(-s,\beta+s,N/2,\frac{|x|^2}{1+|x|^2}\Big),
\end{align*}
where $\mathcal{H}$ denotes the hypergeometric function, and we have used the associated Pfaff transformation in the last identity.
Moreover, since $\beta > {N}/{2}$, we have
$$
\lim_{z\to 1^-}\frac{\mathcal{H}\big(-s,\beta+s,N/2,z\big)}{(1-z)^{N/2-\beta}} = \gamma < 0,
$$
for an explicit constant $\gamma$ depending on $N$, $\beta$ and $s$ (see, e.g., \cite[Chapters 15.2, 15.4]{DLMF}).
As a consequence, there exist $R_0 \geq 1$ and two positive constants $\eta_1,\,\eta_2$ such that
\begin{equation} \label{eq:Deltaskappageq}
\begin{gathered}
0 < \frac{\eta_1}{(1+|x|^2)^{N/2+s}}
\leq -(-\Delta)^s\kappa_1(x)
\leq \frac{\eta_2}{(1+|x|^2)^{N/2+s}} \\[0.1cm]
\text{for every $x\in\R^N$ with $|x| > R_0$}.
\end{gathered}
\end{equation}
In particular, we deduce that $-(-\Delta)^s\kappa_1(x)\to 0$ as $|x|\to +\infty$, and thus, since $-(-\Delta)^s\kappa_1$ is continuous on $\R^N$ (see Proposition \ref{prop:welldefDeltas}), it follows that
\begin{equation} \label{eq:DeltaskappaBd}
\text{there exists $\mathcal{B} > 0$ such that $|(-\Delta)^s\kappa_1|\leq \mathcal{B}$ pointwise on $\R^N$}.
\end{equation}
Gathering all the above estimates, we can easily complete the verification of property (3).

Indeed, by combining \eqref{eq:Deltakappageq} and \eqref{eq:Deltaskappageq},
for every $x\in\R^N$ with $|x| > R_0$ we have
\begin{align*}
-\LL\kappa_1(x)+\lambda\kappa_1(x) & = \mathbf{a}\Delta\kappa_1(x)-\mathbf{b}(-\Delta)^s\kappa_1(x)+\lambda\kappa_1(x) \\
& \geq -\frac{2\mathbf{a}\beta N}{\mathbf{c}_\beta(1+|x|^2)^{\beta+2}}	+\frac{\lambda}{\mathbf{c}_\beta(1+|x|^2)^\beta} \\
& = \frac{1}{\mathbf{c}_\beta(1+|x|^2)^{\beta}}\big[\lambda-2\mathbf{a}\beta N(1+|x|^2)^{-2}\big] \\
& \geq \frac{1}{\mathbf{c}_\beta(1+|x|^2)^{\beta}}\big[\lambda-2\mathbf{a}\beta N(1+R_0^2)^{-2}\big] \geq 0,
\end{align*}
provided that
$$\lambda \geq \lambda_1 = 2\mathbf{a}\beta N(1+R_0^2)^{-2}>0.$$
On the other hand, for every $x\in\R^N$ with $|x|\leq R_0$, by
\eqref{eq:DeltakappaBd} and \eqref{eq:DeltaskappaBd} we have
\begin{align*}
	-\LL\kappa_1(x)+\lambda\kappa_1(x) & = \mathbf{a}\Delta\kappa_1(x)-\mathbf{b}(-\Delta)^s\kappa_1(x)+\lambda\kappa_1(x) \\
	&\geq -\mathbf{a}\cdot\mathcal{A}-\mathbf{b}\cdot\mathcal{B}+\frac{\lambda}{\mathbf{c}_\beta(1+|x|^2)^\beta} \\
	& \geq -\mathbf{a}\cdot\mathcal{A}-\mathbf{b}\cdot\mathcal{B}+\frac{\lambda}{\mathbf{c}_\beta(1+R_0^2)^\beta}  > 0,
\end{align*}
provided that
$$\lambda\geq \lambda_2 = \mathbf{c}_\beta(1+R_0^2)^\beta\cdot\big(\mathbf{a}\cdot\mathcal{A}+\mathbf{b}\cdot\mathcal{B}\big)> 0.$$
Summing up, we conclude that $\kappa_1$ fulfills property (3)
\emph{for every $\lambda\geq \lambda_0$}, where
\begin{equation} \label{eq:deflambdazero}
\lambda_0 =  \max\{\lambda_1,\lambda_2\}> 0.	
\end{equation}
In other words, $\kappa_1$ satisfies \eqref{eq:Prop3strongerform} (with the above choice of $\lambda_0$).
\vspace{0.1cm}

Having completed the proof in the case $\varepsilon = 1$, we now consider the general case $0<\varepsilon<1$, exploiting the scaling properties of $\Delta$ and $(-\Delta)^s$.

To this end, we make the following key observation: by definition,
\begin{equation} \label{eq:kappekappa1}
\kappa_\varepsilon(x) = \varepsilon^{N/2}\kappa_1(\varepsilon^{1/2}x)
\qquad \forall\, x\in\mathbb{R}^N.
\end{equation}
In view of \eqref{eq:kappekappa1}, it follows immediately that $\kappa_\varepsilon$ is smooth and strictly positive in $\mathbb{R}^N$, since the same properties hold for $\kappa_1$. Hence, $\kappa_\varepsilon$ satisfies property (1) in Definition \ref{def:KaplanFunction}.

Moreover, since $0<\varepsilon<1$, for every $x\in\mathbb{R}^N$ we have
\begin{align*} \kappa_\e(x)+|\nabla \kappa_\e(x)| & = \e^{N/2}\big(\kappa_1(\e^{1/2}x)+\e^{1/2}|(\nabla \kappa_1)(\e^{1/2}x)|\big) \\ & \leq \e^{N/2}\big(\kappa_1(\e^{1/2}x)+|(\nabla \kappa_1)(\e^{1/2}x)|\big) \\ & (\text{since $\kappa_1$ fulfills \eqref{eq:ExtraKaplanIntbyParts}}) \\ & \leq c\,\e^{N/2}\Psi_\beta(\e^{1/2}x) \leq c\,\e^{(N-1)/2}\Psi_\beta(x), \end{align*}As a consequence, $\kappa_\varepsilon$ satisfies \eqref{eq:ExtraKaplanIntbyParts}, and therefore also property (2) in Definition \ref{def:KaplanFunction} (see \eqref{eq:kappekappa1} and recall that $\|\kappa_1\|_{L^1(\mathbb{R}^N)}=1$ by the choice of $\mathbf{c}_\beta$).

In order to verify property (3), we exploit the scaling properties of $\Delta$ and $(-\Delta)^s$. Indeed, taking into account \eqref{eq:kappekappa1},
 for every $x\in\RN$ we have
\begin{align*}
	& -\LL\kappa_\e(x)+\lambda\kappa_\e(x) = \mathbf{a}\Delta\kappa_\e(x)-
	\mathbf{b}(-\Delta)^s\kappa_\e(x)+\lambda\kappa_\e(x) \\
	& \qquad = (\e\mathbf{a})(\Delta\kappa_1)(\e^{1/2}x)-(\e^{s}\mathbf{b})[(-\Delta)^s\kappa_1](\e^{1/2}x)+\lambda\kappa_1(\e^{1/2}x);
\end{align*}
from this, since we have already demonstrated that $-\LL\kappa_1+\lambda\kappa_1\geq 0$ on $\RN$ for every $\lambda\geq \lambda_0$, and since we have an \emph{explicit} dependence of $\lambda_0$ with respect to the coefficients
$\mathbf{a},\mathbf{b}$, we derive that
$-\LL\kappa_\e+\lambda\kappa_\e\geq 0$ pointwise on $\R^N$, provided that
$$\lambda\geq \lambda_0(\e) = \max\big\{2(\e \mathbf{a})\beta N(1+R_0^2)^{-2},\mathbf{c}_\beta(1+R_0^2)^\beta\cdot\big((\e \mathbf{a})\cdot\mathcal{A}+(\e^s\mathbf{b})\cdot\mathcal{B}\big)\big\}.$$
On the other hand, recalling that $0<\e < 1$ (hence, $\e\leq \e^s$), we have
$$\lambda_0(\e)\leq \e^s\lambda_0,$$
from which we infer that $\kappa_\e$ verify property (3) in the stronger form
\eqref{eq:Prop3strongerform} (with $\lambda_0$ as in \eqref{eq:deflambdazero}). This completes the demonstration.
\end{proof}
\begin{remark}
	We stress that the previous Theorem \ref{thm:ExplicitKaplan} is the only result which does not hold in this form if $\mathbf{a}>0$ and $\mathbf{b}=0$ (i.e. in the purely local case). It is known indeed (see, e.g., \cite{BL, deP} that in this case the right Kaplan function is of the form 
	\begin{equation*}
		\left(\dfrac{k}{\pi}\right)^{N/2} \mathbf{e}^{-k|x|^2},
	\end{equation*}
	\noindent for some $k>0$.
\end{remark}	
By combining Theorems \ref{thm:MainNonExistenceKaplan}-\ref{thm:ExplicitKaplan},
we obtain the following `more explicit' result.
\begin{theorem} \label{thm:NonexistenceExplicit}
	Let $u_0\in C_+(\mathbb{R}^N)\cap L^\infty(\mathbb{R}^N)$, and assume that there exist $\beta > N/2,\,\e\in(0,1]$ such that the following
	estimate holds
	\begin{equation} \label{eq:conditionKaplanExplicit}
		\frac{\e^{N/2}}{\mathbf{c}_\beta}\int_{\R^N}\frac{u_0(x)}{(1+\e|x|^2)^\beta}\,dx >\mathbf{s}_f(\e^s\lambda_0)
	\end{equation}
	\emph{(}where $\lambda_0 > 0$ as in \eqref{eq:deflambdazero}, and
	$\mathbf{c}_\beta = \|\Psi_\beta\|_{L^1(\R^N)}$\emph{)}. Then, the Cauchy problem \eqref{eq:mainPbParabolicIntro}, with initial datum $u_0$, does not admit global classical solution.
\end{theorem}
\begin{proof}
	On account of Theorem \ref{thm:ExplicitKaplan}, we know that the function
	$$\kappa_\e(x) = \frac{\e^{N/2}}{\mathbf{c}_\beta}\frac{1}{(1+\e|x|^2)^\beta}$$
	is a Kaplan function for  $\LL$ (since $\beta > N/2$ and $\e\in(0,1]$), and it also satisfies \eqref{eq:ExtraKaplanIntbyParts}. Moreover, it fulfills property (3) in De\-fi\-nition \ref{def:KaplanFunction} \emph{for every  $\lambda \geq \e^s\lambda_0$} (with $\lambda_0 > 0$
	as in \eqref{eq:deflambdazero}).
	
	The conclusion then follows directly
	from Theorem \ref{thm:MainNonExistenceKaplan} by choosing
	$$\kappa = \kappa_\e \quad\text{and}\quad \lambda = \e^s\lambda_0,$$
    since in this case condition \eqref{eq:KaplanConditionuzero} reduces
    to \eqref{eq:conditionKaplanExplicit}. This ends the proof.
	\end{proof}
We conclude this section by showing how Theorem \ref{thm:NonexistenceExplicit} can be used to recover the non-exi\-stence of global classical solutions to \eqref{eq:mainPbParabolicIntro}, for every non-zero initial datum, in the case
$$
f(z) = z^p, \qquad 1<p<p_F = 1+\frac{2s}{N}.
$$
Such a result was already established in \cite{BPV} by different methods and for the broader class of very weak solutions (see also ??).

\begin{theorem}\label{teo3}
Let $f(z) = z^p$, with $1<p<p_F$. Then, the Cauchy problem \eqref{eq:mainPbParabolicIntro}
\emph{(}associated with this $f$\emph{)} does not admit global classical solutions \emph{for every non-zero initial datum}
$$u_0\in C_+(\mathbb{R}^N)\cap L^\infty(\mathbb{R}^N).$$
\end{theorem}
\begin{proof}
	Let $u_0\in C_+(\mathbb{R}^N)\cap L^\infty(\mathbb{R}^N),\,u_0\not\equiv 0$. By Theorem \ref{thm:NonexistenceExplicit}, it suffices to prove that there exist $\beta > N/2$ and $\varepsilon\in(0,1]$ such that\begin{equation}\label{eq:toProveSubcritical}
\begin{split}
& \frac{\e^{N/2}}{\mathbf{c}_\beta}\int_{\R^N}\frac{u_0(x)}{(1+\e|x|^2)^\beta}\,dx >\mathbf{s}_f(\e^s \lambda_0)\\
&\qquad\qquad\Longleftrightarrow\,\,
	\int_{\R^N}\frac{u_0(x)}{(1+\e|x|^2)^\beta}\,dx >\mathbf{c}_\beta \lambda_0^{1/(p-1)}\e^{s/(p-1)-N/2},
\end{split}	
\end{equation}
where we have used  that $\mathbf{s}_f(\lambda) = \lambda^{1/(p-1)}$ when $f(z) = z^p$ (see \eqref{eq:sflambdapower}).
\vspace{0.1cm}

Fix now $\beta > N/2$. On the one hand, since $u_0\geq 0$ and $u_0\not\equiv 0$, by the Monotone Convergence Theorem we have the following computation
$$\lim_{\e\to 0^+}\int_{\R^N}\frac{u_0(x)}{(1+\e|x|^2)^\beta}\,dx =
\int_{\RN}u_0(x)\,dx > 0;$$
on the other hand, since $1<p<p_F$, we have
$$\lim_{\e\to 0^+}\e^{s/(p-1)-N/2} = 0.$$
As a consequence, since both $\mathbf{c}_\beta$ and $\lambda_0$ are  independent of $\e$, there exists $0<\e^*<1$ such that
condition \eqref{eq:toProveSubcritical} is satisfied. This ends the proof.
\end{proof}

\end{document}